\documentclass[a4paper, reqno]{amsart}
\usepackage{amssymb, amsmath, amsthm, chngcntr, enumerate, mathabx, mathrsfs, mathtools, tikz, url}
\usepackage[numbers]{natbib}

\usetikzlibrary{arrows,calc}

\numberwithin{equation}{section}
\newtheorem{thm}{Theorem}

\newtheorem{lemma}[thm]{Lemma}
\newtheorem{proposition}[thm]{Proposition}

\newtheorem{rmk}[thm]{Remark}

\newtheorem*{theorem*}{Theorem}

\newcommand{\R}{\mathbb{R}}
\newcommand{\Z}{\mathbb{Z}}
\newcommand{\N}{\mathbb{N}}
\newcommand{\T}{\mathbb{T}}

\title[Mapping properties of the Littlewood-Paley square function]{Endpoint Mapping properties of the Littlewood-Paley square function}

\author{Odysseas Bakas}
\address{Room 4606, James Clerk Maxwell Building, University of Edinburgh, Peter Guthrie Tait Road, Edinburgh, EH9 3FD.}
\email{o.bakas@sms.ed.ac.uk}

\date{}

\begin{document}

\begin{abstract} In this note we give an alternative proof of a theorem due to Bourgain \cite{Bourgain} concerning the growth of the constant in the Littlewood-Paley inequality on $\T$ as $p \rightarrow 1^+$. Our argument is based on the endpoint mapping properties of Marcinkiewicz multiplier operators, obtained by Tao and Wright in \cite{TW}, and on Tao's converse extrapolation theorem \cite{Tao}. Our method also establishes the growth of the constant in the Littlewood-Paley inequality on $\T^n$ as $p \rightarrow 1^+$. Furthermore, we obtain sharp weak-type inequalities for the  Littlewood-Paley square function on $\T^n$, but when $n \geq 2$ the weak-type endpoint estimate on the product Hardy space over the $n$-torus fails, contrary to what happens when $n=1$.
\end{abstract}

\maketitle

\section{Introduction}
If $f$ is a trigonometric polynomial on the torus $\T$, then the Littlewood-Paley square function $S(f)$ of $f$ is given by
$$ S(f) (x) = \Big( \sum_{k \in \Z} |\Delta_k (f)(x)|^2  \Big)^{1/2},$$
where for $k\in \N$ one defines
$$ \Delta_k (f)(x) = \sum_{n=2^{k - 1}}^{2^k -1} \widehat{f} (n) e^{ i 2 \pi n x} \ \mathrm{and}\  \Delta_{-k} (f) (x) =  \sum_{n=-2^k +1 }^{-2^{k-1}} \widehat{f} (n) e^{ i 2 \pi n x}   $$
and $\Delta_0 (f) (x) = \widehat{f}(0) $ for $x \in \T$. 

The square function $S$ can be extended as a bounded operator on $L^p (\T)$ for all $1<p< \infty$, namely for each $1<p<\infty$ there is a constant $C(p)$ so that
\begin{equation}\label{L-P}
\| S (f) \|_{L^p (\T)} \leq C(p) \| f \|_{L^p(\T)}.
\end{equation}
In \cite[Theorem 1]{Bourgain}, Bourgain determined the behaviour of $C(p)$ in (\ref{L-P}) as $p \rightarrow 1^+$. In particular, he showed that there exist absolute constants $c_1, c_2 > 0$ such that
\begin{equation}\label{estimate}
c_1 (p-1)^{-3/2} < C(p) < c_2 (p-1)^{-3/2}
\end{equation}
for every $1 < p \leq 2$.

In section \ref{alternative} we give a simple proof of the upper estimate in (\ref{estimate}) based on results of Tao and Wright \cite{TW} and Tao \cite{Tao}. More precisely, using the observation that Marcinkiewicz multipliers locally map $L \log^{1/2} L $ to $L^{1,\infty}$ \cite[Theorem 1.2]{TW}, together with interpolation and Tao's converse extrapolation \cite{Tao}, one deduces that $ \|  \sum_{k \in \Z  } \pm \Delta_k \|_{L^p (\T) \rightarrow L^p (\T)} \lesssim (p-1)^{-3/2} $, which is essentially the upper estimate in (\ref{estimate}). Furthermore, we extend (\ref{estimate}) to higher dimensions. Indeed, by using  $ \|  \sum_{k \in \Z } \pm \Delta_k \|_{L^p (\T) \rightarrow L^p (\T)} \lesssim (p-1)^{-3/2} $ and iteration, we obtain higher-dimensional extensions of (\ref{estimate}) in section \ref{extension}. In section \ref{sharp_w-t} we prove sharp weak-type inequalities for the multi-parameter Littlewood-Paley square function on $\T^n$ and in section \ref{sharp_rough} we establish the corresponding weak-type endpoint results on $\R^n$. It is well-known that the Littlewood-Paley square function maps $H^1 (\T)$ to $L^{1, \infty}(\T)$. Motivated by this fact, a natural question is whether the two-parameter Littlewood-Paley square function maps the product real Hardy space $H^1 (\T \times \T)$ to $L^{1, \infty} (\T^2)$. In section \ref{negative_results} we show that this is not the case.

\subsection*{Acknowledgement} The author wishes to warmly thank and acknowledge his PhD supervisor, Professor Jim Wright, for all his help, guidance and for his useful comments that improved the presentation of this paper.
\section{Notation}
\subsection{Notation and useful facts} 
If $X$ and $Y$ are positive quantities, the notation $X \lesssim Y$ (or $Y \gtrsim X$) means that there is a positive constant $C>0$ such that $X \leq C Y$. If $X \lesssim Y$ and $Y \lesssim X$, we write $X \sim Y$.

Let $(X,\mu)$ be a measure space and $r>0$. We set
$$ \| f \|_{L \log^r L (X)} = \int_0^{\infty} f^{\ast} (t) \log^r(1/t) dt ,$$
 where $f^{\ast} (t) = \inf \big\{ \lambda > 0 : \mu(\{ x \in X : |f(x)| >\lambda \} ) \leq t \big\} $ is the decreasing rearrangement of $f$ defined on $[0,\infty)$. If $\mu(X) < \infty$, then $\int_X |f(x)| \log^r(1+ |f(x)|) d \mu(x) \lesssim 1  +\| f \|_{L \log^r L (X)}$ and $\| f \|_{L \log^r L (X)} \lesssim 1 +\int_X |f(x)| \log^r(1+ |f(x)|) d \mu(x)$. See, e.g., \cite{BS}.

In the present note, we identify functions on $\T$ with functions on $[0,1)$.

If $K_n$ denotes the Fej\'{e}r kernel on $\T$ of order $n$, then $V_n = 2 K_{2n+1} - K_n$ is the de la Vall\'{e}e Poussin kernel of order $n$. Since $\| K_n \|_{L^1 (\T)} =1$ and $\| K_n \|_{L^{\infty} (\T)} \lesssim  n$ for every $n \in \N$, we deduce that $\int_{\T} |V_n (x)|  \log^r (1 + |V_n(x)|) dx \lesssim \log^r n$ for $r >0$. Moreover, one has $\widehat{V_n} (j) =1$ for all $|j| \leq n+1$ and it thus follows 
 that $\| \Delta_k (V_{2^N}) \|_{L^1 (\T)} \gtrsim k $
for each $k \in \N$ with $k \leq N$. 

Let $(X, \mu)$ be a given measure space. One has (see page 485 in \cite{RDF})
\begin{equation}\label{weak_sup}
 \| g \|_{L^{1,\infty}(X)} \sim \sup_{\substack{E \subset X :\\ 0 < \mu(E) <\infty}} \mu(E)^{-1} \| g \|_{L^{1/2}(E)}.
\end{equation}

\subsection{Hardy spaces} We define the real Hardy space $H^1 (\R)$ to be the space of all integrable functions whose Hilbert transforms are also integrable.

The product real Hardy space $H^1 (\R \times \R)$ is the set of all functions $f \in L^1 (\R^2)$ such that $H_1 (f), H_2 (f) , H_1 \otimes H_2 (f) \in L^1 (\R^2)$, where $H_i$ denotes the Hilbert transform with respect to the $i$-th variable.

 Let $R = I \times J $ be a dyadic rectangle in $\R^2$. Following \cite{CF}, we say that a function $a_R$ is a rectangle atom associated to $R$, if $a_R$ is supported in $R$, $\| a_R \|_{L^2 (\R^2)} \leq |R|^{-1/2}$, $\int_I a_R (x',y) dx'=0$ for every $y \in J$ and $\int_J a_R (x,y') dy' =0$ for every $x \in I$. We define $H^1_{\mathrm{rect}} (\R \times \R)$ to be the space spanned by the class of all rectangle atoms, namely $ H^1_{\mathrm{rect}} (\R \times \R) =\big\{ \sum_{R} \lambda_R a_R : a_R\ \mathrm{is} \ \mathrm{a}\ \mathrm{rectangle}\ \mathrm{atom} \ \mathrm{and}\ \sum_R |\lambda_R| < \infty \big\} $. A counterexample of Carleson \cite{Carleson} shows that $H^1_{\mathrm{rect}} (\R \times \R)$ is a proper subspace of $H^1 (\R \times \R)$.

Similarly, by using the periodic Hilbert transform, one defines the real Hardy space $H^1 (\T)$ and the product real Hardy space $H^1 (\T \times \T)$. One defines rectangle atoms $a_R$ associated to dyadic rectangles $R \subset \T^2$ as in the euclidean case. In the periodic setting, one also needs to consider constant functions on $\T^2$ and ``essentially one-dimensional atoms'', that is functions defined on $\T^2$ that are constant in one variable $a(x,y) = a_I(x) $ (or $a(x,y)=a_I (y)$), $a_I$ is supported in a dyadic interval $I \subset \T$, has mean zero and $\| a_I \|_{L^2 (\T)} \leq |I|^{-1/2} $. We define $H^1_{\mathrm{rect}} (\T \times \T)$ to be the space spanned by rectangle atoms associated to dyadic rectangles in $\T^2$, constant functions on $\T^2$ and ``essentially one-dimensional atoms'' on $\T^2$. Furthermore, $H^1_{\mathrm{rect}} (\T \times \T)$ is a proper subspace of $H^1 (\T \times \T)$.

For more details on Hardy spaces, see the survey articles \cite{CF} and \cite{CW}.

\section{A new proof of the upper estimate in (\ref{estimate})}\label{alternative}
In \cite{TW}, Tao and Wright proved that if $T_m$ is a Marcinkiewicz multiplier operator acting on functions defined over $\R$ (namely the corresponding symbol $m$ of $T_m$ is a bounded function on $\R$ and $m$ is of uniform bounded variation over intervals of the form $\pm [2^k, 2^{k+1})$, $k \in \Z$), then it locally maps $L \log^{1/2} L$ to $L^{1, \infty}$. In particular, for every compact set $K \subset \R$ there is a constant $C > 0$, depending on $K$ and $\| m \|_{L^{\infty}(\R)} + \sup_{k \in \Z} \int_{\pm [2^k, 2^{k+1})} |dm|$, such that
\begin{equation}\label{weak-type}
\| T_m (f) \|_{L^{1,\infty} (K)} \leq C \| f \|_{L \log^{1/2}L (K)}
\end{equation} 
for all measurable functions $f$ supported in $K$.

By adapting the proof of Tao and Wright for functions defined on the torus, one can show that for every $\omega \in [0,1]$ the prototypical Marcinkiewicz multiplier operator
$$ T_{\omega} = \sum_{k \in \Z  } r_k (\omega) \Delta_k$$ 
acting on functions defined over $\T$ maps $L \log^{1/2} L (\T)$ to $L^{1, \infty} (\T)$, where $(r_k)_{k \in \Z }$ denotes the set of Rademacher functions indexed by $\Z $. In particular, one has
\begin{equation}\label{w-t}
\| T_{\omega} (f) \|_{L^{1, \infty} (\T)} \leq C \| f \|_{L \log^{1/2}L (\T)},
\end{equation}
where $C>0$ is an absolute constant independent of $\omega$.

Using (\ref{w-t}) and the fact that $T_{\omega}$ is bounded on $L^2 (\T)$ with operator norm equal to $1$, one can easily show, by using  Marcinkiewicz-type interpolation, that $T_{\omega}$ is bounded from $L \log^{3/2} L (\T)$ to $L^1 (\T)$. In particular, we obtain
$$ \| T_{\omega} (f) \|_{L^1 (\T)} \lesssim \| f \|_{L \log^{3/2}L (\T)}, $$
where the implied constant is independent of $\omega$. By Tao's converse extrapolation theorem \cite{Tao}, it follows that
\begin{equation}\label{L^p_Est}
\| T_{\omega} \|_{L^p(\T) \rightarrow L^p (\T)} \leq \frac{A}{(p-1)^{3/2}} \ \ (\mathrm{as}\ p \rightarrow 1^+)
\end{equation}
where $A>0$ is a positive constant independent of $\omega$. To complete the proof of the upper estimate in (\ref{estimate}), we use (\ref{L^p_Est}) and Khintchine's inequality. More precisely, let $p >1$ be close to $1$ and let $f$ be a trigonometric polynomial. Then, by Khintchine's inequality, we have for every $x \in \T$
$$ S(f) (x) \lesssim \int_{[0,1]} |T_{\omega} (f)(x)| d \omega,$$
where the implied constant is independent of $x \in \T$ and $f$. Therefore, by integrating over $\T$ and using Minkowski's inequality, we get
\begin{align*}
 \| S(f)\|_{L^p (\T)} \lesssim \Big( \int_{\T} \Big| \int_{[0,1]} |T_{\omega} (f)(x)| d \omega \Big|^p  dx \Big)^{1/p}& \leq \int_{[0,1]} \|  T_{\omega} (f) \|_{L^p (\T)} d \omega \\
&\lesssim  \int_{[0,1]} \frac{1}{(p-1)^{3/2}} \| f \|_{L^p (\T)} d \omega \\
&= \frac{1}{(p-1)^{3/2}} \| f \|_{L^p (\T)},
\end{align*}
which is the upper estimate in (\ref{estimate}).
\section{Higher dimensional extension of (\ref{estimate})}\label{extension}

For $n \in \N$, let $S_n$ denote the $n$-parameter Littlewood-Paley square function on $\T^n$ initially defined over trigonometric polynomials on $\T^n$ by
$$ S_n (f) (x) = \Big( \sum_{k_1, \cdots, k_n \in \Z }  | \Delta_{k_1, \cdots, k_n} (f) (x) |^2 \Big)^{1/2},$$
where $ \Delta_{k_1, \cdots, k_n} =  \Delta_{k_1} \otimes \cdots \otimes \Delta_{ k_n}$. The corresponding $n$-parameter Littlewood-Paley inequality is
\begin{equation}\label{L-P_n}
 \|S_n (f) \|_{L^p (\T^n)} \leq C_p (n)  \|f \|_{L^p (\T^n)}.
\end{equation}

Our goal in this section is to show that $C_p (n) \sim (p-1)^{-3n/2}$. As it is mentioned in the introduction, this can be done quite easily by iteration thanks to the fact that $ \| T_{\omega} \|_{L^p (\T) \rightarrow L^p (\T)} \lesssim (p-1)^{-3/2}$.

\begin{proposition}\label{constants_d}
There exist positive constants $c_1 (n), c_2(n)$, depending only on the dimension $n$, such that
\begin{equation}\label{main_n}
 \frac{c_1(n)}{(p-1)^{3n/2}} <  C_p (n)  < \frac{c_2(n)}{(p-1)^{3n/2}},
\end{equation}
where $C_p (n)$ is the constant in $(\ref{L-P_n})$.
\end{proposition}

\begin{proof} To obtain the upper estimate in (\ref{main_n}), let $\omega_1, \cdots,\omega_n$ be arbitrary numbers in $[0,1]$. Then, by using (\ref{L^p_Est}) and iteration, we deduce that 
$$\| T_{\omega_1} \otimes \cdots \otimes T_{\omega_n} \|_{L^p (\T^n) \rightarrow L^p (\T^n)} \leq \frac{ A^n}{(p-1)^{3n/2}}.$$
As in the one-dimensional case, by using multi-dimensional Khintchine's inequality (see, e.g., Appendix D in \cite{Stein}) and Minkowski's inequality, we obtain
$$  \|S_n (f) \|_{L^p (\T^n)} \leq \frac{c_2(n)}{(p-1)^{3n/2}}  \|f \|_{L^p (\T^n)}, $$
where $c_2 (n)$ is a constant that depends only on $n \in \N$.

To prove the lower estimate, we use the corresponding argument of Bourgain that shows the lower estimate in (\ref{estimate}). As in \cite{Bourgain}, given $p>1$, take $N \in \N$ to be such that $\log N \sim (p-1)^{-1}$ and set $f =V_N $. Since $\| S(V_N)\|_{L^p (\T)} \gtrsim (p-1)^{-3/2}$, we have
$$ \| S_n (V_N \otimes \cdots \otimes V_N ) \|_{L^p (\T^n)} = \| S (V_N) \|_{L^p (\T)} \cdots \| S(V_N) \|_{L^p (\T)} \gtrsim_n (p-1)^{-3n/2}, $$
as desired. \end{proof}

It is worth noting that by adapting the method presented in section \ref{alternative} one can give an alternative proof to the upper estimate in (\ref{main_n}). In particular, one can first study the endpoint mapping properties of $n$-dimensional Marcinkiewicz multiplier operators of the form $T_{\omega_1} \otimes \cdots \otimes T_{\omega_n}$ and then, one can use converse extrapolation to deduce the growth of $C_p (n)$ as $p \rightarrow 1^+$ (see also remark \ref{second_proof}). The advantage of this indirect approach is that it motivates the study of sharp weak-type inequalities for $S_n$, which can be regarded as a rudimentary prototype of general Marcinkiewicz multipliers in higher dimensions. This is a problem interesting in its own right.

\section{Sharp weak-type estimates for the Littlewood-Paley square function on $\T^n$}\label{sharp_w-t}
\subsection{The one-dimensional case}
Assume that for some $r\geq 0$ the Littlewood-Paley square function $S$ satisfies a weak-type inequality of the form
$$  \| S(f)  \|_{L^{1,\infty} (\T)} \leq C \| f \|_{L \log^r L (\T)} $$
over all trigonometric polynomials $f$ on $\T$, where $C>0$ is some absolute constant. We shall prove that necessarily $r \geq 1/2$. For this, note that by using the above inequality and the fact that $S$ is bounded on $L^2 (\T)$, we deduce, by interpolation, that
$$ \|S(f) \|_{L^1 (\T)} \lesssim  \| f \|_{L \log^{r+1} L (\T)} $$
for all trigonometric polynomials $f$ on $\T$. However, if we take $f = V_{2^N}$, then we have $ \| f \|_{L \log^r L (\T)} \lesssim 1+ \int_{\T} |f| \log^{r+1}(1+|f|) \lesssim N^{r+1}$ and, moreover, by Minkowski's inequality,
\begin{align*}
\| S(f) \|_{L^1 (\T)} \geq \Big\|  \Big( \sum^N_{k=1} |\Delta_k (V_{2^N})  |^2 \Big)^{1/2} \Big\|_{L^1 (\T)} \geq \Big( \sum_{k=1}^N \| \Delta_k (V_{2^N}) \|^2_{L^1 (\T)} \Big)^{1/2}  &\gtrsim \big( \sum_{k=1}^N  k^2 \big)^{1/2} \\
&\gtrsim N^{3/2}.
\end{align*}
We thus get  $N^{3/2} \lesssim N^{r+1}$ and hence, by letting $N \rightarrow \infty$, it follows that the best we can expect is $r \geq 1/2$. 

\begin{proposition} The Littlewood-Paley square function $S$ satisfies the weak-type inequality
\begin{equation}\label{w-t_S}
\| S(f) \|_{L^{1,\infty}(\T)} \leq C \| f \|_{L \log^{1/2} L (\T)} 
\end{equation}
for all trigonometric polynomials $f$ on $\T$, where $C>0$ is an absolute constant.
\end{proposition}

\begin{proof} This follows immediately from the work of Tao and Wright \cite[Theorem 1.2]{TW}. In particular, (\ref{w-t_S}) can be regarded as a vector-valued version of (\ref{w-t}).

 More precisely, to prove (\ref{w-t_S}), let $f$ be a fixed trigonometric polynomial on $\T$. Note that for every measurable subset $E$ of $\T$ with $|E| > 0$, by Khintchine's inequality and Fubini's theorem, there is a choice of $\omega' \in [0,1]$, depending on $f$ and $E$, such that $ \|  T_{\omega'}(f) \|_{L^{1/2}(E)}  \gtrsim  \| S(f) \|_{L^{1/2}(E)} $. Hence, (\ref{w-t_S}) follows from (\ref{weak_sup}) and (\ref{w-t}).
\end{proof}
\subsection{The higher-dimensional case} In this paragraph we extend (\ref{w-t_S}) to higher dimensions, namely we obtain weak-type estimates for the $n$-parameter Littlewood-Paley square function $S_n$. To do this, as in the one-dimensional case, we reduce the problem to the study of the corresponding mapping properties of certain randomised analogues of $S_n$, namely we study first the mapping properties of Marcinkiewicz multiplier operators of the form $T_{\omega_1} \otimes \cdots \otimes T_{\omega_n}$ on $\T^n$, where $\omega_i \in [0,1]$. To this aim, notice that whenever $T$ is a linear operator acting on functions defined over some measure space $(X, \mu)$ with $\mu(X)=1$, that is bounded on $L^2(X)$  and bounded from $L \log^{1/2} L (X)$ to $L^{1, \infty} (X)$, then it is bounded from $L \log^{r+3/2} L (X)$ to $L \log^r L(X)$ for every $r \geq 1/2$. In particular, one has
\begin{equation}\label{Log_map}
 \int_X |T(f) (x)| \log^r (1 + |T(f)(x)|) d \mu(x) \lesssim 1  + \int_X |f(x)| \log^{r+\frac{3}{2}} (1 + |f(x)|) d\mu(x) .
\end{equation}

Using (\ref{Log_map}) and induction, one can easily establish sharp weak-type estimates for Marcinkiewicz multiplier operators of the form $T_{\omega_1} \otimes \cdots \otimes T_{\omega_n}$ on $\T^n$,  $\omega_i \in [0,1]$.

\begin{lemma}\label{w-k_tensor}
Let $n \in \N$ be a given dimension.

For $\omega_1, \cdots, \omega_n \in [0,1]$ consider the $n$-dimensional Marcinkiewicz multiplier operator $T_{\omega_1} \otimes \cdots \otimes T_{\omega_n}$, where $T_{\omega_i}$ is as in section $\ref{alternative}$.

Then the operator $T_{\omega_1} \otimes \cdots  \otimes T_{\omega_n}$ maps $L \log^{a_n} L (\T^n)$ to $L^{1,\infty} (\T^n)$, where $a_n = 1/2 + 3(n-1)/2$, and in particular,
\begin{equation}\label{weak_map_prop}
\| T_{\omega_1} \otimes \cdots  \otimes T_{\omega_n} (f) \|_{L^{1,\infty} (\T^n)} \lesssim 1 + \int_{\T^n} |f| \log^{a_n} (1+ |f|) .
\end{equation}
\end{lemma}

\begin{proof} We proceed by induction on $n \in \N$. The case $n=1$ corresponds to (\ref{w-t}).

Assume now that for some integer $n>1$ the desired inequality (\ref{weak_map_prop}) holds. To obtain the $(n+1)$-dimensional case, fix an arbitrary $\alpha >0$ and some $f $ in $L\log^{a_{n+1}} L(\T^{n+1})$. Then, by using Fubini's theorem, we may write
\begin{align*}
& | \big\{ (x_1, \cdots, x_{n+1}) \in \T^{n+1} : | T_{\omega_1} \otimes \cdots \otimes T_{\omega_{n+1}} (f) (x_1, \cdots, x_{n+1}) | > \alpha \big\}| = \\
& \int_{\T} | \big\{ (x_1, \cdots x_n) \in \T^n : |   T_{\omega_1} \otimes \cdots \otimes T_{\omega_n} (T_{\omega_{n+1}} (f) ) (x_1, \cdots, x_{n+1}) | > \alpha \big\} | d x_{n+1}.
\end{align*}
Hence, by our inductive hypothesis and Fubini's theorem,
\begin{align*}
& \alpha | \big\{ (x_1, \cdots, x_{n+1}) \in \T^{n+1} : | T_{\omega_1} \otimes \cdots \otimes T_{\omega_{n+1}} (f) (x_1, \cdots, x_{n+1}) | > \alpha \big\}| \lesssim \\
& 1+ \int_{\T^n } \big[ \int_{\T} | T_{\omega_{n+1}} (f)  (x_1, \cdots, x_{n+1}) | \log^{a_n}(1 + | T_{\omega_{n+1}} (f)  (x_1, \cdots, x_{n+1}) | )  d x_{n+1} \big] d x_1 \cdots d x_n
\end{align*}
and so, by (\ref{Log_map}) and Fubini's theorem,
\begin{align*}
& \alpha | \big\{ (x_1, \cdots, x_{n+1}) \in \T^{n+1} : | T_{\omega_1} \otimes \cdots \otimes T_{\omega_{n+1}} (f) (x_1, \cdots, x_{n+1}) | > \alpha \big\}| \lesssim \\
& 1 +  \int_{\T^n } \big[ \int_{\T} |  f  (x_1, \cdots, x_{n+1}) | \log^{a_n +3/2}(1 + | f  (x_1, \cdots, x_{n+1}) |)   d x_{n+1} \big] d x_1 \cdots d x_n = \\
&  1 + \int_{\T^{n+1}} |f(x_1 ,\cdots, x_{n+1})| \log^{ a_n + 3/2} (1+|f(x_1, \cdots, x_{n+1})|) dx_1 \cdots d x_{n+1} .
\end{align*}
Since $a_{n+1} = a_n + 3/2$, the proof of the lemma is complete. \end{proof}

Now, an adaptation of the argument used in the one-dimensional case gives the main result of this paragraph.

\begin{proposition}
For any given $n \in \N$, there is a constant $C_n>0$ such that the $n$-parameter Littlewood-Paley square function satisfies the weak-type inequality
\begin{equation}\label{w-t_n}
\| S_n (f) \|_{L^{1,\infty} (\T^n)} \leq C_n \big[ 1 + \int_{\T^n} |f| \log^{a_n} (1+ |f |) \big],
\end{equation}
for all trigonometric polynomials $f$ on $\T^n$, where $a_n = 1/2 +3(n-1)/2$. Moreover, the exponent $a_n$ in $(\ref{w-t_n})$ is sharp.
\end{proposition}

\begin{proof} As in the one-dimensional case, we use Khintchine's inequality and (\ref{weak_sup}) to show that there exists a choice of $\omega_1', \cdots, \omega_n' \in [0,1]$ such that
$$ \| S_n (f) \|_{L^{1,\infty} (\T^n)} \lesssim_n  \| T_{\omega_1'} \otimes \cdots \otimes T_{\omega_n' } (f) \|_{L^{1,\infty} (\T^n)}.$$
Hence, by using (\ref{weak_map_prop}), we deduce that $S_n$ satisfies the desired weak-type inequality (\ref{w-t_n}).

To prove that the exponent $a_n$ in (\ref{w-t_n}) cannot be improved, assume that the inequality holds for some $r \geq 0$. Since $S_n$ is bounded on $L^2 (\T^n )$, it follows by interpolation that $S_n$ satisfies
$$ \| S_n (f) \|_{L^1(\T^n)} \lesssim 1 + \int_{\T^n} |f(x_1, \cdots, x_n)| \log^{r+1} (1+ |f ( x_1, \cdots, x_n)|) dx_1 \cdots dx_n .$$
If we take $f$ to be $V_{2^N} \otimes \cdots \otimes V_{2^N} $, then 
$$ \| S_n (f) \|_{L^1(\T^n)} = \| S(V_{2^N}) \|_{L^1(\T)} \cdots \| S(V_{2^N})\|_{L^1 (\T)} \gtrsim N^{3n/2}$$
but $\int_{\T^n} |f| \log^{r+1} (1+ |f|) \lesssim N^{r+1}$. Hence, by letting $N \rightarrow \infty$, we see that we must have $r \geq -1 + 3n/2 = a_n$.
\end{proof}

\begin{rmk}\label{second_proof}
As it is mentioned in section $\ref{extension}$, by using Lemma $\ref{w-k_tensor}$, interpolation, and converse extrapolation (as in the one-dimensional case), one can give an alternative proof of Proposition $\ref{constants_d}$.
\end{rmk}

\section{Endpoint mapping properties of the multi-parameter rough Littlewood-Paley square function in the euclidean case}\label{sharp_rough}
If $f$ is a Schwartz function on $\R$, we define its rough Littlewood-Paley square function $S_{\R} (f)$ by
$$ S_{\R} (f) (x) = \Big(  \sum_{k \in \Z} |P_k (f) (x)|^2 \Big)^{1/2}, $$
where $(P_k f)^{\widehat{\ }} (\xi) =  \chi_{[2^k, 2^{k+1})} (\xi)  \widehat{f} (\xi) +  \chi_{(-2^{k+1}, -2^k]} (\xi)  \widehat{f} (\xi) $ is the rough Littlewood-Paley projection at frequencies $|\xi| \sim2^k$, $k \in \Z$. For $n \in \N$, the $n$-parameter rough Littlewood-Paley square function is given by
$$ S_{\R^n} (f) (x) = \Big(  \sum_{k_1, \cdots, k_n \in \Z} |P_{k_1} \otimes \cdots \otimes P_{k_n}  (f) (x)|^2 \Big)^{1/2} $$
for $f$ initially belonging to the class of Schwartz functions on $\R^n$.

In the following proposition we show that the $n$-parameter rough Littlewood-Paley square function on $\R^n$ satisfies weak-type inequalities analogous to the ones obtained in the previous section, if we restrict ourselves to compacts subsets of $\R^n$.
 
\begin{proposition}
For any given $n \in \N$ and each compact set $K $ in $\R^n$, there is a constant $C_{K,n}>0$ such that the $n$-parameter Littlewood-Paley square function satisfies the weak-type inequality
\begin{equation}\label{w-t_rough}
 \| S_{\R^n} (f) \|_{L^{1, \infty} (K)} \leq C_{K,n}  \big[ 1+  \int_K |f| \log^{a_n} (1+|f|)  \big]
\end{equation} 
for each measurable function $f$ supported in $K$, where $a_n = 1/2 + 3(n-1)/2$. Moreover, the exponent $a_n$ in $(\ref{w-t_rough})$ is sharp.
\end{proposition}

\begin{proof} The argument that establishes (\ref{w-t_rough}) is similar to the one given in the previous section, where one uses (\ref{weak-type}) instead of (\ref{w-t}).

It remains to prove sharpness. Consider the one-dimensional case first. For this, assume that for some $r\geq 0$ one has
$$  \| S_{\R } (f) \|_{L^{1, \infty} ([-1,1])} \lesssim   1+  \int_{[-1,1]} |f| \log^r (1+|f|)    $$
for every measurable function $f$ supported in $[-1,1]$. By interpolation, we deduce that
\begin{equation}\label{inter}  
\| S_{\R } (f) \|_{L^1 ([-1,1])} \lesssim  1+  \int_{[-1,1]} |f| \log^{r+1} (1+|f|) 
\end{equation}
for all measurable functions $f$ with $\mathrm{supp}(f) \subset [-1,1]$. To show that $r \geq a_1 = 1/2$, let $N $ be a large positive integer to be chosen later and let $\phi $ be a fixed Schwartz function such that $\mathrm{supp}(\phi) \subset [-2,2]$ and $\phi|_{[-1,1]} \equiv 1$. Define $g(x) = 2^N \widecheck{\phi} (2^N x)$, $x \in \R$. Then $g$ is a Schwartz function satisfying $\| g\|_{L^1 (\R) } \sim 1$, $\| g \|_{L^{\infty} (\R)} \lesssim 2^N$, where the implied constants depend only on $\phi$ and not on $N$. Hence, 
\begin{equation}\label{eq1}
\int_{[-1,1] } |g| \log^{r+1} (1+ |g|) \lesssim N^{r+1}. 
\end{equation} 
Using Minkowski's inequality and the fact that $\widehat{g}|_{[-2^N,2^N]} \equiv 1$ we get
$$ \| S_{\R } (g) \|_{L^1 ([-1,1])} \geq \Big( \sum_{k \in \Z} \| P_k (g) \|^2_{L^1 ([-1,1])} \Big)^{1/2} \geq \Big( \sum_{k=1}^N \| P_k (g) \|^2_{L^1 ([-1,1])} \Big)^{1/2} \gtrsim N^{3/2}. $$
Define $f = g \chi_{[-1,1]}$ and $e =g -f$. One can easily check that, by the construction of $g$, the ``error'' satisfies $\| e \|_{L^2 (\R)} \lesssim 1$. Moreover, $f$ is supported in $[-1,1]$ and $ \| S_{\R } (g) \|_{L^1 ([-1,1])} \leq \sqrt{2} [  \| S_{\R } (f) \|_{L^1 ([-1,1])} +  \| S_{\R } (e) \|_{L^1 ([-1,1])} ] $. By using the Cauchy-Schwarz inequality,
$  \| S_{\R } (e) \|_{L^1 ([-1,1])} \leq \sqrt{2 }  \| S_{\R } (e) \|_{L^2 ([-1,1])} \leq  \sqrt{2 }  \| S_{\R } (e) \|_{L^2 (\R)}$ and since $   \| S_{\R } (e) \|_{L^2 (\R) } = \| e \|_{L^2 (\R)} \lesssim 1$, we deduce that 
\begin{equation}\label{eq2}
  \| S_{\R } (f) \|_{L^1 ([-1,1])} \gtrsim N^{3/2}.
\end{equation} 
Since $|f| \leq |g|$, (\ref{eq1}) implies that 
\begin{equation}\label{eq3} \int_{[-1,1] } |f| \log^{r+1} (1+ |f|)  \lesssim N^{r+1} .
\end{equation} 
Combining (\ref{inter}), (\ref{eq2}) and (\ref{eq3}), we get $N^{3/2} \lesssim N^{r+1}$. Letting $N \rightarrow \infty$, it follows that $r \geq a_1 = 1/2$, as desired.
 
To prove sharpness in the $n$-dimensional case, assume that (\ref{w-t_rough}) holds for some $r \geq 0$ and for $f$ being as above, take $ h = f \otimes \cdots \otimes f$. Then $h$ is supported in $[-1,1]^n$,  $\int_{[-1,1]^n} |h| \log^{r+1} (1 + |h|) \lesssim N^{r+1} $ and
$$ \| S_{\R^n } (h) \|_{L^1 ([-1,1]^n)} = \| S_{\R} (f) \|_{L^1 ([-1,1])} \cdots \| S_{\R} (f) \|_{L^1 ([-1,1])} \gtrsim N^{3n/2}. $$
Therefore, we must have $r \geq a_n = 1/2 + 3(n-1)/2$.
\end{proof}

\section{Negative results}\label{negative_results}
It is well-known that $S_{\R}$ maps $H^1 (\R)$ to $L^{1, \infty} (\R)$. Indeed, one may write 
\begin{equation}\label{identity}
 S_{\R} (f) (x) = \Big( \sum_{k \in \Z} | P_k (f_k) (x)|^2 \Big)^{1/2} , 
\end{equation}
where $f_k = \widetilde{P_k} (f)$ and $\widetilde{P_k}$ denotes the multiplier operator whose corresponding symbol is $\eta(2^{-k} \cdot)$, where $\eta$ is an even Schwartz function supported in $\pm [1/4,4]$ with $\eta|_{[1,2]}\equiv 1$. By \cite[Corollary 2.13 on p.~488]{RDF}, one has
\begin{equation}\label{vector-valued_w-t}
 \big| \big\{ x\in \R:  \Big( \sum_{k \in \Z} | P_k (f_k) (x)|^2 \Big)^{1/2} > \alpha \big\} \big| \leq \frac{C} {\alpha} \Big\|  \Big( \sum_{k \in \Z} |f_k|^2 \Big)^{1/2} \Big\|_{L^1 (\R)} 
\end{equation}
for every $\alpha >0$. Hence, the estimate $ \| S_{\R} (f) \|_{L^{1,\infty} (\R)} \lesssim \| f \|_{H^1 (\R)}$ follows from (\ref{identity}) and the fact that the right-hand side of (\ref{vector-valued_w-t}) is majorised by $A \|f \|_{H^1 (\R)}$, where $A>0$ is a constant that depends only on the choice of $\eta$, see \cite{Stein_multiplicateurs}. Similarly, the Littlewood-Paley square function $S$ maps $H^1 (\T)$ to $L^{1,\infty} (\T)$.

A natural question is whether an analogous weak-type estimate holds for the two-parameter rough Littlewood-Paley square function. In the two-parameter setting, a candidate endpoint function space is the product Hardy space $H^1 (\R \times \R)$. Our next result shows that such an estimate is not possible in the product setting, as $S_{\R^2}$ does not even locally map $H^1_{\mathrm{rect}} (\R \times \R)$ to $L^{1,\infty}(\R^2)$.

\begin{proposition}\label{negative}
The two-parameter rough Littlewood-Paley square function does not locally map $H^1_{\mathrm{rect}} (\R \times \R)$ to $L^{1,\infty}(\R^2)$.
\end{proposition}

\begin{proof} Let $N \geq 5 $ be a large positive integer to be chosen later. Consider the function $a_N (x) = 2^{N-1} e^{i 2 \pi 2^{N-1}x} \chi_{[0,2^{-(N-1)})} (x)$. Note that for $x \neq 0 $ the kernel  of $P_N$ is given by
$$ \int_{2^{N-1}}^{2^N} e^{ i 2 \pi  \xi x } d \xi + \int_{-2^N}^{-2^{N-1}} e^{ i 2 \pi  \xi x } d \xi = \frac{e^{i 2 \pi  2^N x} - e^{i 2 \pi  2^{N-1}x}}{i 2 \pi  x} + \frac{ e^{-i 2 \pi  2^{N-1} x} -  e^{-i 2 \pi 2^N  x}  }{i 2 \pi x} $$ 
and hence, for $ 8 \cdot 2^{-(N-1)} \leq x \leq  1$ one has

\begin{align*}
P_N (a_N) (x) &= \int_{[0,2^{-(N-1)})} 2^{N-1} e^{i 2 \pi 2^{N-1} y}  \frac{e^{i 2 \pi  2^N (x-y)} - e^{i 2 \pi  2^{N-1}(x-y)}}{2 \pi i (x-y)} dy\\ 
&+  \int_{[0,2^{-(N-1)})} 2^{N-1} e^{i 2 \pi 2^{N-1} y}  \frac{e^{-i 2 \pi  2^{N-1} (x-y)} - e^{-i 2 \pi  2^N (x-y)}}{2 \pi i (x-y)} dy \\
&=-\frac{2^{N-1} e^{i 2 \pi 2^{N-1}x}}{i 2 \pi }\int_{[0,2^{-(N-1)})} \frac{1}{x-y} dy + \frac{2^{N-1} e^{i 2 \pi 2^Nx}}{i 2 \pi }\int_{[0,2^{-(N-1)})} \frac{e^{-i 2 \pi 2^{N-1}y}}{x-y} dy \\
&+\frac{2^{N-1} e^{-i 2 \pi 2^{N-1}  x}}{i 2 \pi }\int_{[0,2^{-(N-1)})} \frac{e^{-i 2 \pi 2^N  y}}{x-y} dy - \frac{2^{N-1} e^{-i 2 \pi 2^N x}}{i 2 \pi }\int_{[0,2^{-(N-1)})} \frac{e^{-i 6 \pi   2^{N-1}y}}{x-y} dy \\
&= I_1^{(N)} (x) + I_2^{(N)} (x) + I_3^{(N)} (x) + I_4^{(N)} (x) .
\end{align*}
Note that for each $8 \cdot 2^{-(N-1)} \leq x \leq 1$ one has
$$  |  I_1^{(N)} (x)| = \frac{2^{N-1}}{2 \pi }  \int_{[0,2^{-(N-1)})} \frac{1}{x-y} dy \geq \frac{1} {2 \pi x}.  $$
We shall bound $|  I_2^{(N)} (x)|$, $|I_3^{(N)} (x)| $ and $| I_4^{(N)} (x) | $ from above. To bound $|  I_2^{(N)} (x)|$, we make use of the cancellation of $e^{-i 2 \pi 2^{N-1} y}$ over $[0, 2^{-(N-1)})$,
\begin{align*}
|  I_2^{(N)} (x)| &=  \frac{2^{N-1}}{2 \pi } \Big| \int_{[0,2^{-(N-1)})}  e^{-i 2 \pi 2^{N-1}y}\big( \frac{1}{x-y} - \frac{1}{x-2^{-1} \cdot 2^{-(N-1)}} \big) dy \Big| \\
&\leq\frac{2^{N-1}}{2 \pi }\int_{[0,2^{-(N-1)})} \Big| \frac{2^{-1} \cdot 2^{-(N-1)} -y}{(x-y) (x -2^{-1} \cdot 2^{-(N-1)})}\Big| dy \\
&\leq  \frac{2}{15 \pi x},
\end{align*}
since $x-y \geq x/2$ for all $y \in [0,2^{-(N-1)})$ and $x - 2^{-1} \cdot 2^{-(N-1)} \geq 15x/16$. Similarly, $|  I_3^{(N)} (x)| \leq 2/(15\pi x)$ and $|  I_2^{(N)} (x)| \leq 2/(15 \pi x)$ Therefore,
$$ |P_N (a_N) (x)| \geq |  I_1^{(N)} (x)| - |  I_2^{(N)} (x)| - | I_3^{(N)} (x) | - |  I_4^{(N)} (x)| \geq \frac{1} {10 \pi x}$$
for all $8 \cdot 2^{-(N-1)} \leq x \leq 1$ and hence,
$$  S_{\R^2}  (a_N \otimes a_N) (x,y) \geq | (P_N \otimes P_N) (a_N \otimes a_N) (x,y) | \geq \frac{1}{100 \pi^2 xy}$$
for $(x,y) \in [8 \cdot 2^{-(N-1)},1]^2$. It thus follows that
$$ \| S_{\R^2} (a_N \otimes a_N) \|_{L^{1, \infty} ([0,1]^2)} \gtrsim N.$$
Since $a_N \otimes a_N$ is a rectangle atom, by letting $N \rightarrow \infty$, one deduces that $S_{\R^2}$ does not locally map $H^1_{\mathrm{rect}} (\R \times \R)$ to $L^{1,\infty}(\R^2)$.
 \end{proof}

By adapting the proof of the previous proposition we obtain a corresponding negative result in the periodic setting.

\begin{proposition}
The two-parameter Littlewood-Paley square function $S_2$ does not map $H^1_{\mathrm{rect}} (\T \times \T)$ to $L^{1,\infty} (\T^2)$.
\end{proposition}

\begin{proof}
Let $N \geq 9$ be an integer to be chosen later. For $x \in [0,1)$ we decompose the kernel of $\Delta_N$ as
$$ \sum_{n=2^{N-1}}^{2^N -1} e^{i 2 \pi n x} = \frac{e^{i 2 \pi 2^N x} - e^{i 2 \pi 2^{N-1}x}}{e^{i 2 \pi x} -1} = \beta_N (x) + \gamma_N (x) ,   $$ 
where for $x \in (0,1)$ one has

$$ \beta_N (x) = (e^{i 2 \pi 2^N x} - e^{i 2 \pi 2^{N-1}x}) \big(  \frac{1}{e^{i 2 \pi x} -1} - \frac{1}{i2 \pi x} \big) \ \mathrm{and}\ \gamma_N (x) =  \frac{e^{i 2 \pi 2^N x} - e^{i 2 \pi 2^{N-1}x}}{i2 \pi x}$$
and $\beta_N (0) = 0$, $\gamma_N (0) = 2^{N-1}$. Define $a_N (x) = 2^{N-1} e^{i 2 \pi 2^{N-1}x} \chi_{[0,2^{-(N-1)})} (x)$ for $x \in [0,1)$. Arguing as in the proof of Proposition \ref{negative}, one shows that 
$$|\gamma_N \ast a_N (x)| \geq \frac{ 11}{30 \pi x}$$ 
for all $ 8 \cdot 2^{-(N-1)} \leq x  < 1  $. Using the series expansion of $e^{i 2 \pi x}$ and the fact that $\sin(2  \pi x ) \geq 4 x $ for every $ 0 \leq x\leq 2^{-2}$, one obtains $|\beta_N (x)| \leq \pi e^{\pi/2}$ for all $ 0 \leq x \leq 2^{-2}$. Since $\|a_N \|_{L^1 (\T)} =1$, it follows that $| \beta_N \ast a_N (x)| \leq \pi e^{ \pi/2}$ for every $2^{-(N-1)} \leq x \leq 2^{-2}$. Therefore, for each $ 8 \cdot 2^{-(N-1)} \leq x \leq 2^{-8}$ one has
$$ |\Delta_N (a_N) (x)| \geq |\gamma_N \ast a_N (x)| - |\beta_N \ast  a_N (x)| > \frac{11}{30 \pi x} -  \pi e^{ \pi/2} \geq \frac{11}{30 \pi x} - \frac{1}{16x} \sim \frac{1}{x}. $$
 Since we may regard $a_N \otimes a_N$ as an atom of $H^1_{\mathrm{rect} }(\T \times \T)$, by letting $N \rightarrow \infty$, we deduce that $S_2 $ does not map $H^1_{\mathrm{rect}} (\T \times \T)$ to $L^{1,\infty} (\T^2)$.
\end{proof}

\bibliographystyle{plainnat}
\bibliography{b}

\end{document}